\documentclass[11pt]{amsart}
\usepackage{amsfonts,amssymb,amsthm,eucal}

\newtheorem{thm}{Theorem}
\newtheorem{cor}[thm]{Corollary}

\newtheorem{prop}[thm]{Proposition}

\newcommand{\R}{\mathbb{R}}
\newcommand{\E}{\mathbb{E}}
\newcommand{\Prob}{\mathbb{P}}

\newcommand{\C}{\mathbb{C}}
\newcommand{\CM}{\mathcal{C}}

\DeclareMathOperator{\Cov}{Cov}

\renewcommand{\Re}{\operatorname{Re}}
\renewcommand{\Im}{\operatorname{Im}}
\newcommand{\abs}[1]{\left\vert #1 \right\vert}

\allowdisplaybreaks[4]
\numberwithin{equation}{section}

\title[Random circulant matrices]{Some results on random circulant matrices}

\author{Mark W.\ Meckes}

\address{Department of Mathematics, Case Western Reserve University,
Cleveland, Ohio 44106, U.S.A.}

\email{mark.meckes@case.edu}

\begin{document}


\begin{abstract}
This paper considers random (non-Hermitian) circulant matrices, and
proves several results analogous to recent theorems on non-Hermitian
random matrices with independent entries. In particular, the limiting
spectral distribution of a random circulant matrix is shown to be
complex normal, and bounds are given for the probability that a
circulant sign matrix is singular.
\end{abstract}

\maketitle


\section{Introduction}
Given a sequence $X_0, X_1, \dotsc$ of independent complex random
variables, denote by $\CM_n$ the $n\times n$ random circulant matrix
with first row $X_0,\dotsc, X_{n-1}$:
\[
\CM_n = \begin{bmatrix}
X_0 & X_1 & \cdots & \cdots & X_{n-2} & X_{n-1} \\
X_{n-1} & X_0 & X_1 & & & X_{n-2} \\
\vdots & X_{n-1} & X_0 & &&  \vdots \\
\vdots & & & \ddots & & \vdots \\
X_2 & & & & X_0 & X_1 \\
X_1 & X_2 & \cdots & \cdots & X_{n-1} & X_0
\end{bmatrix}.
\]
It is well-known (and easy to verify) that the eigenvalues of $\CM_n$
are
\[
\sum_{j=0}^{n-1} \omega_n^{jk} X_j,
  \quad k=0,\dotsc,n-1,
\]
where $\omega_n = e^{2\pi i/n}$. The eigenvalues of random circulant
matrices have only recently been studied explicitly in the literature,
e.g.\ in \cite{BM,BoseSen,MMS}, which consider real symmetric
circulant matrices and variants of them. Other models of random
matrices lying in some linear subspace of matrix space have also been
studied recently in \cite{BoseSen,BDJ,Chatterjee,HM,MMS}, among
others. These papers mainly prove results analogous to classical
theorems for Wigner-type random matrices, i.e.\ symmetric (or
Hermitian) random matrices whose entries are all independent except
for the symmetry constraint.

This paper mainly considers the eigenvalues of the random circulant
matrices $\CM_n$ with no symmetry constraint, and prove results
analogous to recent theorems about random matrices with all
independent entries. In Section \ref{S:LSD} we investigate the
limiting spectral distribution of large-dimensional random circulant
matrices. In Section \ref{S:Gaussian} we consider the joint
distribution of eigenvalues of circulant matrices with Gaussian
entries and observe some consequences, especially for the
distributions of extreme eigenvalues. Finally, in Section
\ref{S:singularity} we investigate the probability that a circulant
matrix with $\pm 1$ Bernoulli entries is singular.

\section{Limiting spectral distribution}\label{S:LSD}

We denote by
\begin{equation}\label{E:lambda}
\lambda_k=\frac{1}{\sqrt{n}}\sum_{j=0}^{n-1} \omega_n^{jk} X_j
\end{equation}
for $k=0,\dotsc, n-1$ the eigenvalues of $n^{-1/2}\CM_n$, and by
\[
\mu_n = \frac{1}{n}\sum_{k=0}^{n-1} \delta_{\lambda_k}
\]
the empirical spectral measure of $n^{-1/2} \CM_n$.  Theorems
\ref{T:sm-ip} and \ref{T:sm-as} show that as $n\to \infty$, $\mu_n$
converges to a universal limiting measure, namely the standard complex
Gaussian measure $\gamma_\C$ with density $\frac{1}{\pi}
e^{-\abs{z}^2}$ with respect to Lebesgue measure on $\C$. Theorem
\ref{T:sm-ip} shows that this convergence holds in probability under
very weak assumptions on the random variables $X_j$ (including in
particular the case of i.i.d.\ random variables with finite variance);
Theorem \ref{T:sm-as} strengthens this to almost sure convergence
under much stronger assumptions.

Theorems \ref{T:sm-ip} and \ref{T:sm-as} are analogues for circulant
matrices of the circular law for eigenvalues of random matrices with
independent entries, which was recently proved by Tao and Vu in
\cite{TVK} in the almost sure sense for i.i.d.\ entries with finite
variance.  It is likely that the conclusion of Theorem \ref{T:sm-as}
also holds in this level of generality.  However, the result of
\cite{TVK} and other results leading up to it were made possible by
recent advances in controlling how close a random matrix with i.i.d.\
entries is to being singular. As discussed in Section
\ref{S:singularity} below, random circulant matrices are frequently
much more likely to be singular; thus it may be difficult to prove the
optimal result for circulant matrices. (On the other hand, the fact
that circulant matrices are normal may make such a result approachable
via the classical moment method used extensively for Hermitian random
matrices.)

\medskip

\begin{thm}\label{T:sm-ip}
Suppose that the $X_j$ satisfy
\begin{equation}\label{E:sm-moments}
\E X_j = 0, \quad \E X_j^2 = \alpha, \quad \E \abs{X_j}^2 = 1,
\end{equation}
for some $\alpha\in \C$ and 
\begin{equation}\label{E:sm-Lindeberg}
\lim_{n\to\infty}\frac{1}{n} \sum_{j=0}^{n-1} 
  \E \big(\abs{X_j}^2 1_{\abs{X_j}>\varepsilon \sqrt{n}}\big) = 0
\end{equation}
for every $\varepsilon > 0$. Then $\mu_n$ converges in expectation and
weak-$*$ in probability to $\gamma_\C$, the standard complex Gaussian
measure on $\C$.
\end{thm}

Observe that the hypotheses cover the case of i.i.d.\ complex random
variables with finite variance, as well as real random variables and
rotationally invariant distributions satisfying the Lindeberg
condition \eqref{E:sm-Lindeberg}.

\begin{proof}
Without loss of generality we may assume $\alpha \in \R$.
Observe that for a measurable set $A\subseteq \C$,
\[
\E \mu_n(A) = \frac{1}{n} \sum_{k=0}^{n-1} \Prob[\lambda_k \in A]
=\frac{1}{n} \sum_{k=0}^{n-1} \Prob\left[\frac{1}{\sqrt{n}}
  \sum_{j=0}^{n-1} \omega_n^{jk} X_j \in A\right].
\]
We consider $\lambda_k$ as a sum of independent random vectors in
$\R^2\cong \C$, so we will need 
\[
\Cov(\omega_n^{jk} X_j) =
  \begin{pmatrix} \E (\Re \omega_n^{jk}X_j)^2 
    & \E (\Re \omega_n^{jk}X_j)(\Im \omega_n^{jk}X_j) \\
     \E (\Re \omega_n^{jk}X_j)(\Im \omega_n^{jk}X_j) 
     & \E (\Im \omega_n^{jk}X_j)^2
  \end{pmatrix}.
\]
The identities
\begin{equation}\begin{split}\label{E:Re-Im}
\tfrac{1}{2} (w+\overline{w}) z &= (\Re w)(\Re z) + i (\Re w)(\Im z),\\
\tfrac{1}{2} (w-\overline{w}) z &= -(\Im w)(\Im z) + i (\Re w)(\Im z).
\end{split}\end{equation}
for $w,z\in \C$ are useful. Letting $w=z=\omega_n^{jk}X_j$,
\[
\sum_{j=0}^{n-1} \E (\Re \omega_n^{jk}X_j)^2 
  = \frac{1}{2}\Re \sum_{j=0}^{n-1} \E (\omega_n^{2jk}X_j^2 + \abs{X_j}^2)
  = \frac{1}{2}\Re \sum_{j=0}^{n-1} (\alpha \omega_n^{2jk} + 1).
\]
Since $\omega_n^{2k}$ is an $n$th root of unity, $\sum_{j=0}^{n-1}
\omega_n^{2jk}=0$ unless $\omega_n^{2k}= 1$, which is the case only if
$k=0$ or $k=n/2$. Thus unless $k=0$ or $n/2$,
\[
\frac{1}{n}\sum_{j=0}^{n-1} \E (\Re \omega_n^{jk})^2 = \frac{1}{2}.
\]
The other covariances are computed similarly and it follows that for
$k\neq 0,n/2$
\[
\frac{1}{n}\sum_{j=0}^{n-1} \Cov(\omega_n^{jk} X_j) = \frac{1}{2} I_2.
\]

Since $\vert\omega_n^{jk}X_j\vert = \abs{X_j}$, by
\eqref{E:sm-Lindeberg} we can now apply a quantitative two-dimensional
version of Lindeberg's central limit theorem to the complex random
variables $\{ \omega_n^{jk}X_j \mid 0\le j \le n-1\}$. Let $A
\subseteq \C$ be measurable and convex and assume $k\neq 0,n/2$. By
the proof of \cite[Corollary 18.2]{BR}), there is a function $h(n)$
with $\lim_{n\to\infty} h(n)=0$, which depends on $A$ and the rate of
convergence in \eqref{E:sm-Lindeberg} but is independent of $k$, such
that
\[
\abs{ \Prob \left[ \frac{1}{\sqrt{n}} \sum_{j=0}^{n-1} \omega_n^{jk}
    X_j \in A \right] - \gamma_\C(A)} \le h(n).
\]
Therefore
\begin{equation}\label{E:M1}
\abs{ \E \mu_n(A) - \gamma_\C(A) } \le \frac{1}{n} \sum_{k=0}^{n-1}
\abs{ \Prob\big[ \lambda_k \in A \big] - \gamma_\C(A) } 
\le \frac{2+(n-2)h(n)}{n} \xrightarrow{n\to \infty} 0.
\end{equation}
Thus the claimed convergence holds in expectation.

\medskip

Next observe that 
\[
\E \mu_n(A)^2 = \E \left[\frac{1}{n^2} \sum_{k,\ell=0}^{n-1}
    1_{\lambda_k\in A} 1_{\lambda_\ell \in A} \right]
= \frac{1}{n^2} \sum_{k,\ell = 0}^{n-1}
    \Prob \big[ (\lambda_k, \lambda_\ell) \in A \times A \big].
\]
As before, we can compute the sums of the relevant covariance matrices
using \eqref{E:Re-Im}, in this case with $w=\omega_n^{jk}X_j$ and
$z=\omega_n^{j\ell}X_j$. We obtain
\[
\frac{1}{n}\sum_{j=0}^{n-1} \Cov(\omega_n^{jk} X_j, \omega_n^{j\ell} X_j)
    = \frac{1}{2}I_4
\]
except when $k=0$ or $n/2$, $\ell=0$ or $n/2$, $k+\ell = 0$ or $n$, or
$k=\ell$.  These exceptional cases account for fewer than $6n$ of the
$n^2$ possible values of $(k,\ell)$. Applying the four-dimensional case
of Lindeberg's theorem similarly to above, we have
\begin{equation}\label{E:M2}
\abs{ \E \mu_n(A)^2 - \gamma_\C(A)^2 } \le 
    \frac{1}{n^2} \sum_{k,\ell=0}^{n-1} \abs{ \Prob\big[
      (\lambda_k,\lambda_\ell) \in A \times A\big] - \gamma_\C(A)^2
    } \xrightarrow{n\to\infty} 0.
\end{equation}
We are of course using here that standard Gaussian measure on $\C^2$
is the two-fold product of $\gamma_\C$.

From \eqref{E:M1} and \eqref{E:M2} it follows that
\[
\E \big[\mu_n(A) - \gamma_\C(A)\big]^2 
    = \big[\E \mu_n(A)^2 - \gamma_\C(A)^2 \big]
      - 2\gamma_\C(A) \big[\E \mu_n(A) - \gamma_\C(A)\big]
    \xrightarrow{n\to\infty} 0.
\]
Thus for every convex measurable $A\subset \C$, the random variable
$\mu_n(A)$ converges to $\gamma_\C(A)$ in $L^2$ and hence in
probability.
\end{proof}

\medskip

Bose and Mitra \cite{BM} proved earlier a result analogous to Theorem
\ref{T:sm-ip} for real symmetric circulant matrices, which is a
circulant matrix analogue of Wigner's semicircle law; the limiting
distribution in this case is a real Gaussian measure. This result was
strengthened from convergence in probability to almost sure
convergence by Massey, Miller, and Sinsheimer \cite{MMS}. The main
result of \cite{BM} is a similar result for circulant Hankel matrices,
which amounts to studying the singular values instead of eigenvalues
of $\CM_n$ (see \cite[Remark 1.2]{BDJ} about the relationship between
Hankel and Toeplitz matrices). Since circulant matrices are normal,
their singular values are simply the moduli of their eigenvalues; thus
this latter result is essentially a corollary of Theorem
\ref{T:sm-ip}. The proof of Theorem \ref{T:sm-ip} follows the basic
outline of the proofs of \cite{BM}, which assume the $X_j$ are i.i.d.\
with finite third absolute moments; the greater generality of Theorem
\ref{T:sm-ip} is achieved by applying Lindeberg's theorem, instead of
the Berry-Esseen theorem as in \cite{BM}.

The statement of Theorem \ref{T:sm-ip} assumes that the same sequence
of random variables $X_j$ is used to construct $\CM_n$ for every
$n$. The proof shows however that the result generalizes directly to
circulant matrices constructed from a triangular array of random
variables.  The same comment applies to Theorem \ref{T:sm-as} below.

\medskip

\begin{thm}\label{T:sm-as}
Suppose that, in addition to \eqref{E:sm-moments} one of the following
holds.
\begin{enumerate}
\item There exists a $K>0$ such that for each $j$, $\abs{X_j}\le K$
almost surely.
\item There exists a $K>0$ such that for each $j$, the distribution of
  $X_j$ satisfies a quadratic transportation cost inequality with
  constant $K$ (see below for the meaning of this).
\end{enumerate}
Then $\mu_n$ converges weak-$*$ to $\gamma_\C$ almost surely.
\end{thm}

Recall that a probability measure $\mu$ on $\R^d$ is said to satisfy a
quadratic transportation cost inequality with constant $K>0$ if
\[
\inf_{\pi\in\Pi(\mu,\nu)} \int\int \abs{x-y}^2 \ d\pi(x,y) 
  \le \sqrt{2KH(\mu | \nu)}
\]
for every probability measure $\nu$ on $\R^d$.  Here $\Pi(\mu,\nu)$ is
the class of probability measures on $\R^d\times\R^d$ with marginals
$\mu$ and $\nu$ respectively, and $H(\mu \vert \nu)$ is relative
entropy. Such an inequality is satisfied in particular if $\mu$
satisfies a logarithmic Sobolev inequality.  See \cite[Chapter 6]{Ledoux}
for background and references.

\begin{proof}
We begin with the assumption that a quadratic transportation cost
inequality is satisfied.  This assumption implies (and is essentially
equivalent to, see \cite{Gozlan}) the following concentration
inequality for Lipschitz functions of $(X_1,\dotsc,X_n)$.  Let
$F:\C^n\to\R$ have Lipschitz constant $\abs{F}_{\mathrm{Lip}}=L$. Then
\begin{equation}\label{E:conc}
\Prob\big[\abs{F(X_1,\dotsc,X_n)-\E F(X_1,\dotsc,X_n)} \ge t\big]
  \le C e^{-ct^2/K^2L^2}
\end{equation}
for every $t>0$, where $c,C>0$ are absolute constants. Now for a
compactly supported smooth function $f:\C\to\R$, define
\[
F(x_1,\dotsc,x_n) = \frac{1}{n}\sum_{j=1}^n f(x_j),
\]
so that $F(X_1,\dotsc,X_n) = \int_\C f \ d\mu_n$. Then 
\[
\abs{F(x_1,\dotsc,x_n)-F(y_1,\dotsc,y_n)} 
  \le \frac{\abs{f}_{\mathrm{Lip}}}{n} \sum_{j=1}^n \abs{x_j-y_j} 
  \le \frac{\abs{f}_{\mathrm{Lip}}}{\sqrt{n}} \sqrt{\sum_{j=1}^n \abs{x_j-y_j}^2},
\]
so $\abs{F}_{\mathrm{Lip}}\le n^{-1/2}
\abs{f}_{\mathrm{Lip}}$. Combining Theorem \ref{T:sm-ip},
\eqref{E:conc}, and the Borel-Cantelli lemma, we obtain that for each
compactly supported smooth function $f:\C\to \R$,
\[
\int_\C f\ d\mu_n \xrightarrow{n \to \infty} \int_\C f\ d\gamma_\C
\]
almost surely. Applying this for a countable dense family of such $f$
proves the theorem.

\medskip

The case of bounded entries is treated similarly using Talagrand's
famous concentration inequality for convex Lipschitz functions of
bounded random variables \cite{Talagrand-ihes} (also see \cite[Section
  4.2]{Ledoux}); in this case an extra step is required to handle
non-convex test functions $f:\C\to\R$.

As above, let $f:\C\to\R$ be compactly supported and smooth.
 Let $R>0$ be
such that $f(x)=0$ for $\abs{x}\ge R$ and let $\lambda\ge 0$ be such
that the eigenvalues of the Hessian $D^2 f$ are bounded below by
$-\lambda$. For $x\in \C$ define
\[
g(x) = \begin{cases} \frac{\lambda}{2}\abs{x}^2 & \mbox{if }\abs{x}\le R,\\
  \lambda R \left(\abs{x} - \frac{R}{2}\right) & \mbox{if }\abs{x} \ge R. 
  \end{cases}
\]
Then $\abs{g}_{\mathrm{Lip}}=\lambda R$ and so
$\abs{f+g}_{\mathrm{Lip}}\le \abs{f}_{\mathrm{Lip}}+\lambda
R$. Furthermore, since $D^2(f+g)\ge 0$, both $g$ and $f+g$ are
convex. Applying the above argument to $f+g$ and $g$ using Talagrand's
concentration inequality in place of \eqref{E:conc} implies that with
probability $1$,
\[
\int_\C (f+g)\ d\mu_n \xrightarrow{n \to \infty} \int_\C (f+g)\ d\gamma_\C
\quad \mbox{and} \quad
\int_\C g\ d\mu_n \xrightarrow{n\to\infty} \int_\C g\ d\gamma_\C
\]
and hence
\[
\int_\C f\ d\mu_n \xrightarrow{n\to\infty} \int_\C f\ d\gamma_\C.
\]
Again, applying this for a countable dense family of such $f$ proves
the theorem.
\end{proof}

\medskip

We remark that a slight generalization of the argument in the last
paragraph shows that convex Lipschitz functions on $\R^d$ form a
convergence-determining class for the family of probability measures
on $\R^d$ with respect to which such functions are integrable. This
fact is presumably well-known to experts but we could not find a
statement of it in the literature.

\section{Gaussian circulant matrices and extreme eigenvalues}\label{S:Gaussian}

The following result is a circulant analogue of classical formulas
(found, e.g., in \cite{Mehta}) for the joint distribution of
eigenvalues of random matrices with independent complex Gaussian
entries.

\medskip

\begin{prop}\label{T:complex-normal}
Let each $X_j$ have the standard complex normal distribution. Then the
sequence $\lambda_0,\dotsc,\lambda_{n-1}$ of eigenvalues of
$n^{-1/2}\CM_n$ is distributed as $n$ independent standard complex
normal random variables.
\end{prop}

\begin{proof}
The map $(X_0,\dotsc,X_{n-1}) \mapsto
(\lambda_0,\dotsc,\lambda_{n-1})$ defined by \eqref{E:lambda} is
easily checked to be a unitary transformation of $\C^n$, so it
preserves the standard Gaussian measure on $\C^n$.
\end{proof}

\medskip

An easy consequence of Proposition \ref{T:complex-normal} is that in
this setting, $\E \mu_n = \gamma_\C$ for every $n$, and not only in
the limit $n\to \infty$ as guaranteed by Theorem \ref{T:sm-ip}.

If each $X_j$ is a standard real normal random variable, then the same
observation implies that the sequence of eigenvalues $\lambda_0,
\dotsc, \lambda_{n-1}$ are jointly Gaussian random variables, but with
singular covariance since this sequence will lie in an $n$-dimensional
real subspace of the $2n$-dimensional space $\C^n$. On the other hand,
Proposition \ref{T:complex-normal} does have the following simple
analogue for complex Hermitian circulant matrices with Gaussian
entries.

\medskip

\begin{cor}\label{T:Hermitian}
Let $\CM_n^H$ be a Hermitian random circulant matrix
\[
\CM_n^H = \begin{bmatrix}
Y_0 & Y_1 & \cdots & \cdots & Y_{n-2} & Y_{n-1} \\
\overline{Y_{n-1}} & Y_0 & Y_1 & & & Y_{n-2} \\
\vdots & \overline{Y_{n-1}} & Y_0 & &&  \vdots \\
\vdots & & & \ddots & & \vdots \\
\overline{Y_2} & & & & Y_0 & Y_1 \\
\overline{Y_1} & \overline{Y_2} & \cdots & \cdots & \overline{Y_{n-1}} & Y_0
\end{bmatrix},
\]
where $Y_0,\ldots,Y_{\lfloor n/2 \rfloor}$ are independent,
$Y_j=Y_{n-j}$ for $j>n/2$, $Y_0$ and $Y_{n/2}$ (if $n$ is even) have
the standard real normal distribution, and $Y_j$ has the standard
complex normal distribution for $1\le j< n/2$. Then the eigenvalues of
$n^{-1/2}\CM_n^H$ are distributed as $n$ independent standard real
normal random variables.
\end{cor}

\begin{proof}
If $\CM_n$ is as in Proposition \ref{T:complex-normal}, then $\CM_n^H$
has the same distribution as $\frac{1}{\sqrt{2}} (\CM_n + \CM_n^*)$.
Because $\CM_n$ is normal, the eigenvalues of $n^{-1/2}\CM_n^H$ are
\[
\frac{1}{\sqrt{2}}(\lambda_k + \overline{\lambda_k}) = 
\sqrt{2} \Re \lambda_k
\]
for $k=0,\dotsc,n-1$, which by Proposition \ref{T:complex-normal}
are independent standard real normal random variables.
\end{proof}

\medskip

The Gaussian Hermitian circulant matrix $\CM_n^H$ of Corollary
\ref{T:Hermitian} should be thought of as a circulant analogue of the
Gaussian Unitary Ensemble, which (up to a choice of normalization) is
defined as $\frac{1}{\sqrt{2}}(G+G^*)$, where $G$ is an $n\times n$
random matrix whose entries are independent standard complex normal
random variables.

\medskip

\begin{cor}\label{E:extremes}
Let $\CM_n$ and $\CM_n^H$ be as in Proposition \ref{T:complex-normal}
and Corollary \ref{T:Hermitian}. Let $\alpha_1 \ge \dotsb \ge \alpha_n
\ge 0$ be the eigenvalues of $n^{-1} \CM_n \CM_n^*$ and let $\beta_1 \ge
\dotsb \ge \beta_n$ be the eigenvalues of $n^{-1/2} \CM_n^H$. Then
$\alpha_n$ has an exponential distribution with mean $1/n$, and
\[
\alpha_1-\log n
\]
and 
\[
\sqrt{2\log n} \left(\beta_1-\sqrt{2\log n} -
\frac{\log\log n + \log 4\pi}{2\sqrt{2\log n}}\right)
\]
both converge in distribution as $n\to\infty$ to the Gumbel
distribution with cumulative distribution function $e^{-e^{-x}}$.
\end{cor}

\begin{proof}
Observe that since $\CM_n$ is normal, the eigenvalues of $n^{-1}\CM_n
\CM_n^*$ are $\abs{\lambda_k}^2$ (though not usually in the same
order), which by Proposition \ref{T:complex-normal} are distributed as
independent exponential random variables. The corollary follows by
combinging this fact and Corollary \ref{T:Hermitian} with classical
theorems of extreme value theory (see \cite[Chapter 1]{LLR}).
\end{proof}

\medskip

The asymptotic distributions of $\beta_1$ and $\alpha_1$ are circulant
matrix analogues of famous results of Tracy and Widom \cite{TW1} and
Johnstone \cite{Johnstone}, respectively. Davis and Mikosch \cite{DM},
who did not consider random circulant matrices explicitly, proved
(with slight modifications) what amounts to a universality result for
$\alpha_1$, which shows the conclusion follows for quite general
distributions of the $X_j$. Following their method Bryc and Sethuraman
\cite{BrycSeth} proved an explicitly stated universality result for
$\beta_1$; these are then the circulant matrix analogues of the
results of Soshnikov in \cite{Soshnikov2,Soshnikov1} respectively. The
rough orders of magnitude of $\alpha_1$ and $\beta_1$ follow in even
greater generality from work of the author \cite{Meckes} and Adamczak
\cite{Adamczak} (see the remarks in \cite[Section 3.1]{Meckes}).

\section{Singularity of circulant sign matrices}\label{S:singularity}

In this section we specialize to the case in which $\Prob[X_j=-1] =
\Prob[X_j=1]=1/2$ for every $j$ and consider the probability that
$\CM_n$ is singular. The corresponding problem for random $n\times n$
matrices $M_n$ with independent $\pm 1$ entries has a long history. An
old conjecture (see \cite{KKS}) claims that
\[
\Prob[M_n \mbox{ is singular}] = \left(\frac{1}{2}+o(1)\right)^n,
\]
which is asymptotically the probability that two rows of $M_n$ are
equal up to sign.  The best result currently known, proved by
Bourgain, Vu, and Wood in \cite{BVW}, is
\[
\Prob[M_n \mbox{ is singular}]\le \left(\frac{1}{\sqrt{2}} + o(1)\right)^n.
\]
By contrast, the following result shows that the singularity
probability of an $n\times n$ random circulant matrix with $\pm 1$
entries depends strongly on the number-theoretic properties of $n$.

\medskip

\begin{thm}\label{T:singularity}
Let $\Prob[X_j = -1]=\Prob[X_j=1]=1/2$ for each $j$. If $n$ is even, then
\begin{equation}\label{E:sing-even}
  \frac{c_1}{\sqrt{n}} \le \Prob[ \CM_n \mbox{ is singular}] 
  \le \frac{c_2}{\sqrt{n}},
\end{equation}
where $c_1,c_2>0$ are absolute constants. If $n\ge 3$ is odd, then
\begin{equation}\label{E:sing-odd}
\Prob[ \CM_n \mbox{ is singular}] \le \min \Bigg\{
  c_3 \frac{d(n)}{n}, \sum_{\substack{1<m\le n \\ m|n}} 2^{-\varphi(m)}\Bigg\},
\end{equation}
where $c_3>0$ is an absolute constant, $d(n)$ is the number of
divisors of $n$, and $\varphi(m)$ is the number of positive integers
less than $m$ which are relatively prime with $m$.
\end{thm}

The lower bound in \eqref{E:sing-even} shows that if $n$ is even, then
an $n\times n$ circulant sign matrix is much more likely to be
singular than a sign matrix with independent entries. The first upper
bound in \eqref{E:sing-odd} implies that the singularity probability is
rather lower if $n$ is odd; it is known that
\begin{equation}\label{E:d(n)}
d(n) \le n^{c/\log \log n},
\end{equation}
for every $n$, see \cite[Theorem 13.12]{Apostol}. The bound is of
course smaller if $n$ has few divisors; for example if $n=p^k$ for an
odd prime $p$ then $d(n)=k+1=\log_p n + 1$.  The second upper bound in
\eqref{E:sing-odd} implies that the singularity probability is extremely
small if $n$ has no small prime factors.  In particular, if $n\ge 3$
is prime then
\[
\Prob[ \CM_n \mbox{ is singular}] = 2^{-n+1}
\]
since this is the probability that all the entries of $\CM_n$ are
equal.  It would be nice to have a more complete description of the
dependence of the singularity probability on the prime factorization
of $n$.

\begin{proof}
Begin by defining the random polynomial
\[
f(t) = \sum_{j=0}^{n-1} X_j t^j,
\]
so that the eigenvalues of $\CM_n$ are $f(\omega_n^k)$ for
$k=0,\dotsc,n-1$. Observe first that $f(1)$ and $f(-1)$ are identically
distributed and 
\begin{equation}\label{E:f(1)}
\Prob[f(1)=0] = \Prob[f(-1)=0] = 
\begin{cases} 2^{-n} \binom{n}{n/2} & \mbox{if $n$ is even},\\
  0 & \mbox{if $n$ is odd}.\end{cases}
\end{equation}
which implies the lower bound in \eqref{E:sing-even} by Stirling's
formula.

For each $k>0$, $\omega_n^k$ is a primitive root of unity of order
$n/m$, where $m=\gcd(k,n)$. The minimal polynomial of $\omega_n^k$
over the rational numbers is thus the cyclotomic polynomial
$\Phi_{n/m}$, so $f(\omega_n^k)=0$ if and only if $\Phi_{n/m}$ is a
factor of $f$. (See e.g.\ \cite[Section V.8]{Hungerford} for
background on cyclotomic polynomials.) Therefore we need only consider
the cases when $k$ is a divisor of $n$, and
\begin{equation}\label{E:union}
\Prob [ \CM_n \mbox{ is singular}] \le \Prob[f(1)=0] +
\sum_{\substack{1\le k < n\\k|n}} \Prob \big[f(\omega_n^k) =
  0\big].
\end{equation}
A straightforward application of the multidimensional inverse
Littlewood-Offord theorem proved by Friedland and Sodin \cite{FS}
yields that
\begin{equation}\label{E:FS}
\Prob \big[f(\omega_n^k)=0\big] \le \frac{c}{n}
\end{equation}
for some $c>0$ when $k\neq 0, n/2$; the calculations involved in
applying the result of \cite{FS} are similar to those in the proof of
Theorem \ref{T:sm-ip}. The upper bound of \eqref{E:sing-even} and the
first upper bound of \eqref{E:sing-odd} follow by combining
\eqref{E:union}, \eqref{E:f(1)}, and \eqref{E:FS}.  To bound the
number of terms in \eqref{E:union} in the even case it is enough to
use the trivial estimate $d(n) < 2\sqrt{n}$ (divisors of $n$ occur in
pairs $k,n/k$ with $k\le \sqrt{n}$) rather than the much more delicate
result \eqref{E:d(n)}.

For $n\ge 3$ and $d\ge 2$ let $P_{n,m}$ be the set of polynomials $g$
with rational coefficients of degree at most $n-1$ such that $\Phi_m$
is a factor of $g$. With the substitution $m=n/k$, \eqref{E:union}
becomes
\[
\Prob [ \CM_n \mbox{ is singular}] \le \sum_{\substack{1< d \le
    n\\d|n}} \Prob \big[ f\in P_{n,m}\big].
\]
Since $\Phi_m$ has degree $\varphi(m)$, $P_{n,m}$ is an
$(n-\varphi(m))$-dimensional subspace of the (rational) vector space
of polynomials of degree at most $n-1$. Therefore some set of
$n-\varphi(m)$ coefficients of a polynomial $g$ suffice to determine
whether $g\in P_{n,m}$, and so
\[
\Prob \big[ f\in P_{n,m}\big] \le 2^{-\varphi(m)},
\]
which proves the second upper bound in \eqref{E:sing-odd}.
\end{proof}


\section*{Acknowledgements}

The author thanks W{\l}odek Bryc and Elizabeth Meckes for helpful
discussions.

\bibliographystyle{plain} \bibliography{circulant}

\begin{thebibliography}{10}

\bibitem{Adamczak}
R.~Adamczak.
\newblock A few remarks on the operator norm of random {T}oeplitz matrices.
\newblock Preprint, available at {\tt http://arxiv.org/abs//0803.3111}.

\bibitem{Apostol}
T.~M. Apostol.
\newblock {\em Introduction to Analytic Number Theory}.
\newblock Springer-Verlag, New York, 1976.
\newblock Undergraduate Texts in Mathematics.

\bibitem{BR}
R.~N. Bhattacharya and R.~Ranga~Rao.
\newblock {\em Normal Approximation and Asymptotic Expansions}.
\newblock Robert E. Krieger Publishing Co. Inc., Melbourne, FL, 1986.
\newblock Reprint of the 1976 original.

\bibitem{BM}
A.~Bose and J.~Mitra.
\newblock Limiting spectral distribution of a special circulant.
\newblock {\em Statist. Probab. Lett.}, 60(1):111--120, 2002.

\bibitem{BoseSen}
A.~Bose and A.~Sen.
\newblock Another look at the moment method for large dimensional random
  matrices.
\newblock {\em Electron. J. Probab.}, 13:no. 21, 588--628, 2008.

\bibitem{BVW}
J.~Bourgain, V.~Vu, and P.~Wood.
\newblock On the singularity probability of discretely random complex matrices.
\newblock Preprint.

\bibitem{BDJ}
W.~Bryc, A.~Dembo, and T.~Jiang.
\newblock Spectral measure of large random {H}ankel, {M}arkov and {T}oeplitz
  matrices.
\newblock {\em Ann. Probab.}, 34(1):1--38, 2006.

\bibitem{BrycSeth}
W.~Bryc and S.~Sethuraman.
\newblock On the maximum eigenvalue for circulant matrices.
\newblock Manuscript in preparation.

\bibitem{Chatterjee}
S.~Chatterjee.
\newblock Fluctuations of eigenvalues and second order {P}oincar\'e
  inequalities.
\newblock {\em Probab. Theory Related Fields}, 143(1-2):1--40, 2009.

\bibitem{DM}
R.~A. Davis and T.~Mikosch.
\newblock The maximum of the periodogram of a non-{G}aussian sequence.
\newblock {\em Ann. Probab.}, 27(1):522--536, 1999.

\bibitem{FS}
O.~Friedland and S.~Sodin.
\newblock Bounds on the concentration function in terms of the {D}iophantine
  approximation.
\newblock {\em C. R. Math. Acad. Sci. Paris}, 345(9):513--518, 2007.

\bibitem{Gozlan}
N.~Gozlan.
\newblock A characterization of dimension free concentration in terms of
  transportation inequalities.
\newblock Preprint, available at {\tt http://arxiv.org/abs/0804.3089}.

\bibitem{HM}
C.~Hammond and S.~J. Miller.
\newblock Distribution of eigenvalues for the ensemble of real symmetric
  {T}oeplitz matrices.
\newblock {\em J. Theoret. Probab.}, 18(3):537--566, 2005.

\bibitem{Hungerford}
T.~W. Hungerford.
\newblock {\em Algebra}, volume~73 of {\em Graduate Texts in Mathematics}.
\newblock Springer-Verlag, New York, 1980.
\newblock Reprint of the 1974 original.

\bibitem{Johnstone}
I.~M. Johnstone.
\newblock On the distribution of the largest eigenvalue in principal components
  analysis.
\newblock {\em Ann. Statist.}, 29(2):295--327, 2001.

\bibitem{KKS}
J.~Kahn, J.~Koml{\'o}s, and E.~Szemer{\'e}di.
\newblock On the probability that a random {$\pm 1$}-matrix is singular.
\newblock {\em J. Amer. Math. Soc.}, 8(1):223--240, 1995.

\bibitem{LLR}
M.~R. Leadbetter, G.~Lindgren, and H.~Rootz{\'e}n.
\newblock {\em Extremes and Related Properties of Random Sequences and
  Processes}.
\newblock Springer Series in Statistics. Springer-Verlag, New York, 1983.

\bibitem{Ledoux}
M.~Ledoux.
\newblock {\em The Concentration of Measure Phenomenon}, volume~89 of {\em
  Mathematical Surveys and Monographs}.
\newblock American Mathematical Society, Providence, RI, 2001.

\bibitem{MMS}
A.~Massey, S.~J. Miller, and J.~Sinsheimer.
\newblock Distribution of eigenvalues of real symmetric palindromic {T}oeplitz
  matrices and circulant matrices.
\newblock {\em J. Theoret. Probab.}, 20(3):637--662, 2007.

\bibitem{Meckes}
M.~W. Meckes.
\newblock On the spectral norm of a random {T}oeplitz matrix.
\newblock {\em Electron. Comm. Probab.}, 12:315--325 (electronic), 2007.

\bibitem{Mehta}
M.~L. Mehta.
\newblock {\em Random Matrices}, volume 142 of {\em Pure and Applied
  Mathematics (Amsterdam)}.
\newblock Elsevier/Academic Press, Amsterdam, third edition, 2004.

\bibitem{Soshnikov1}
A.~Soshnikov.
\newblock Universality at the edge of the spectrum in {W}igner random matrices.
\newblock {\em Comm. Math. Phys.}, 207(3):697--733, 1999.

\bibitem{Soshnikov2}
A.~Soshnikov.
\newblock A note on universality of the distribution of the largest eigenvalues
  in certain sample covariance matrices.
\newblock {\em J. Statist. Phys.}, 108(5-6):1033--1056, 2002.
\newblock Dedicated to David Ruelle and Yasha Sinai on the occasion of their
  65th birthdays.

\bibitem{Talagrand-ihes}
M.~Talagrand.
\newblock Concentration of measure and isoperimetric inequalities in product
  spaces.
\newblock {\em Inst. Hautes \'Etudes Sci. Publ. Math.}, (81):73--205, 1995.

\bibitem{TVK}
T.~Tao and V.~Vu.
\newblock Random matrices: Universality of esds and the circular law.
\newblock Preprint, available at {\tt http://arxiv.org/abs/0807.4898}. With an
  appendix by M. Krishnapur.

\bibitem{TW1}
C.~A. Tracy and H.~Widom.
\newblock Level-spacing distributions and the {A}iry kernel.
\newblock {\em Comm. Math. Phys.}, 159(1):151--174, 1994.

\end{thebibliography}

\end{document}